\documentclass[11 pt]{amsart}
\usepackage{comment}
\usepackage{enumerate}
\usepackage{amssymb, amsmath}
\usepackage{bm}


\def\GL{\mathrm{GL}(2,\mathbb{Z})}

\def\Z{\mathbb{Z}}
\def\F{\mathbb{F}}
\def\Q{\mathbb{Q}}
\def\disc{\mathrm{disc}}

\begin{document}
	
	\newtheorem{thm}{Theorem}[section]
	\newtheorem{prop}[thm]{Proposition}
	\newtheorem{lemma}[thm]{Lemma}
	\newtheorem{claim}[thm]{Claim}
	\newtheorem{cor}[thm]{Corollary}
	\newtheorem{conj}[thm]{Conjecture}
	\newtheorem{rk}[thm]{Remark}
	\newtheorem{defi}[thm]{Definition}

	\title[Thue equations that simultaneously fail the Hasse principle]{Thue equations that simultaneously fail the Hasse principle}

	\author{Paloma Bengoechea}
	\address{ ETH, Mathematics Dept.\\CH-8092, Z\"urich, Switzerland}
	\email{paloma.bengoechea@math.ethz.ch}
	\thanks{Bengoechea's research is supported by SNF grant 173976.}

	\begin{abstract}
		We refine a previous construction by Akhtari and Bhargava so that, 
		for every positive integer $m$, we obtain a positive proportion of  Thue equations $F(x,y)=h$ 
		that fail  the integral Hasse principle simultaneously for every positive integer $h$ less than $m$.
		The  binary forms $F$ have fixed degree $\geq 3$ and are  ordered by the absolute value of the maximum of the coefficients.
		
	\end{abstract}
	
	\maketitle
	
	\section{Introduction}
	
	Let $F(x, y)$ be an irreducible binary form  with integer coefficients of degree $n\geq 3$, and let $m$ be a positive integer.  
	Thue showed that the inequalities
	\begin{equation}\label{1}
		|F(x , y)| \leq m,
	\end{equation}
	known as \emph{Thue inequalities},  have  finitely many solutions in integers $x$ and $y$. Mahler \cite{Mah9}, Schmidt \cite{S87}, Mueller-Schmidt 
	\cite{MS88}, Thunder \cite{Thu}, recently the author \cite{Ben}, and others obtained  upper bounds for the number of solutions. 
	
	In this article we will prove that, when the forms are ordered by height, a positive proportion 
	have  no integer solutions $(x,y)$ to \eqref{1} but have $p$-adic solutions for all primes $p$ to all the equations $F(x,y)=h$, where $h$ 
	is any positive integer less than $m$. 
	
	This article is a refinement of \cite{AB}, where a positive proportion of  Thue equations $F(x,y)=m$ are shown to  have solutions
	locally everywhere but not globally (i.e. they fail the integral Hasse principle). 
	We refine the construction in \cite{AB} so that the Hasse principle fails simultaneously for all Thue equations $F(x,y)=h$, where 
	$1\leq h\leq m$.

	\begin{thm}\label{th1}
		Let $m$ be a positive integer. When integral binary forms of fixed degree $\geq 3$  are ordered by height, a positive proportion of  
		primitive maximal forms   represent every positive integer less than $m$ locally everywhere but  not  globally.
	\end{thm}


	A \textit{maximal} form $F$ is a form which is not a proper subform of any other form, i.e. $F$ cannot be written as $G(ax+by,cx+dy)$ for any integer 
	matrix $A=\tiny{\begin{pmatrix}a &b\\c &d\end{pmatrix}}$ with $|\det(A)|>1$ and any other binary form $G$. In Theorem \ref{th1} we exclude non-maximal forms because
	if $F$ is a proper subform of $G$, then the number of integer solutions of $F$ and $G$ is the same.
	
	The  construction in \cite{AB} also works  for cubic forms  ordered by discriminant instead of the height, and for quartic 
	forms  ordered by the two generators of their rings of 
	invariants, classically denoted by $I,J$ (see \cite{Ak}).  
	The same can be done here; Theorem \ref{th1} holds for cubic forms ordered by discriminant and for quartic forms ordered by the invariants $I,J$.

	The paper is organized as follows. In  section \ref{P} we give some preliminaries, specially a result due to Gy\"ory (Theorem \ref{Gyory}) that gives
	an upper bound for 
	the number of solutions to Thue inequalities \eqref{1}. 
	In section \ref{sec F} we explain our strategy and
	how Gy\"ory's result will be relevant for the construction of forms satisfying Theorem \ref{th1}. 
	We also give an expicit lower bound for the proportion in Theorem \ref{th1} that depends on $m$ and the fixed degree  of the forms.
	In section \ref{secGj} we construct those forms.
	We prove that  they represent  every positive integer less than $m$ locally everywhere, but not globally, in sections \ref{sec local} and \ref{sec global} 
	respectively.

	\section{Preliminaries}\label{P}
	
	\subsection{ Discriminant, height and equivalent forms}
	
	For a binary form $F(x , y)$ that factors over $\mathbb{C}$ as
	$$
	\prod_{i=1}^{n} (\alpha_{i} x - \beta_{i}y),
	$$
	the discriminant $D(F)$ of $F$ is given by
	$$
	D(F) = \prod_{i<j} (\alpha_{i} \beta_{j} - \alpha_{j}\beta_{i})^2.
	$$
	Therefore, if we write
	$$
	F(x, y) = a_n (x - \gamma_{1}y)\ldots  (x- \gamma_{n}y),
	$$ 
	we have
	$$
	D(F) = a_n ^{2(n-1)} \prod_{i<j} (\gamma_{i}  - \gamma_{j})^2.
	$$
	
	
	If we write  $F(x , y)=a_{n}x^{n} + a_{n-1}x^{n-1} y+ \ldots + a_{1}x y^{n-1} + a_{0}y^n$,
	the (naive) height of $F$, denoted by $H(F)$, is defined by 
	$$
	H(F) =\max  \left( |a_{n}|, |a_{n-1}|, \ldots , |a_{0}|\right).
	$$


	For $A=\tiny{\begin{pmatrix} a & b \\ c & d \end{pmatrix}}$ and a binary form $F(x,y)$, we define the binary form $F_{A}$ by
	$$
	F_{A}(x , y) = F(ax + by , cx + dy).
	$$
	Note that
	\begin{equation}\label{St6}
		D(F_{A}) = (\textrm{det}(A))^{n (n-1)} D(F).
	\end{equation}
	If $A\in\GL$ we say that  $F_A$ is equivalent to $F$, and if $A\in\mathrm{GL}(2,\mathbb{Z}_p)$ we say that $F_A$ is equivalent to $F$ over $\Z_p$.
	Two forms that are equivalent (respectively equivalent over $\Z_p$) represent the same integers (respectively the same $p$-adic integers).

	\subsection{Number of solutions to Thue inequalities}
	
	
	The upper bounds for the number of solutions to Thue inequalities \eqref{1} obviously depend on $m$. However, if $m$ is bounded above in terms
	of the absolute value of the discriminant of $F$, then one can find bounds that are independent of $m$.
	The following is Corollary 3 of \cite{Gyory} with 
	$\vartheta=1/3$
	\footnote{Note that similar results to Theorem \ref{Gyory} were given previously in  \cite{EG} for the number of all solutions, not necessarily primitive.}
	.
	
	\begin{thm}[Gy\"ory]\label{Gyory}
		Let $F(x,y)\in\mathbb{Z}[x,y]$ be a  binary form  of degree $n \geq3$ and $m$ be an integer satisfying
		\begin{equation}\label{cGyory}
			|D(F)|\geq (3^n m)^{6(n-1)}.
		\end{equation}
		The number of  solutions to $|F(x,y)|\leq m$ in co-prime integers $x,y$ is at most $32n+11$. 
	\end{thm}
	
	A solution $(x,y)$ to \eqref{1} is called \emph{primitive} if $x$ and $y$ are co-prime integers.
	
	\subsection{Finding solutions of Thue equations over $\Z_p$}
	
	In order to find solutions of Thue equations or inequalities over $\Z_p$, most of the time we will  find a solution over $\F_p$ and lift it to 
	$\Z_p$ using Hensel's Lemma:
	If $(\alpha,\beta)\in\Z_p^2$ satisfies $F(\alpha,\beta)\equiv0\pmod{p}$ and 
	either $F_x'(\alpha,\beta)\not\equiv 0$ or $F'_y(\alpha,\beta)\not\equiv 0$ $\pmod{p}$, then there exists $(\alpha_0,\beta_0)\in\Z_p^2$ such that
	$(\alpha_0,\beta_0)\equiv(\alpha,\beta)\pmod{p}$ and $F(\alpha_0,\beta_0)=0$ over $\Z_p$.

	\section{Strategy}\label{sec F}

	Let $m\geq 1$, $n\geq 3$ and set $g=\frac{(n-1)(n-2)}{2}$. Let 
	$$
	M= \max(m,(2g+1)^2,17)
	$$
	and denote by $p_1,\ldots,p_k$ all the primes less than $M$. We have that
	\begin{equation}\label{2k}
		2^k-32n-11\geq 1. 
	\end{equation}
	Indeed, for $n\geq4$,
	$$
	k\geq \dfrac{M}{\log M}\geq \dfrac{(2g+1)^2}{2\log(2g+1)}\geq 32n+11,
	$$
	and for $n=3$, 
	$$
	2^k\geq 2^{7}\geq 32\cdot3+12.
	$$
	Let $X$ be a positive integer such that
	$$
	X> \dfrac{2^{(n+1)k} \prod_{i=1}^kp_i^{n-1}}{(n+1)^{k/2}},
	$$
	and define 
	$$
	\tilde X := X\,\dfrac{(n+1)^{k/2}}{2^{(n+1)k}\prod_{i=1}^kp_i^{n-1}} \geq1.
	$$

	Let $S$ be the set of maximal primitive integral binary forms of degree $n$ that satisfy the conditions: 
	\begin{enumerate}[(I)]
		\item they represent every positive integer $\leq m$ locally everywhere;
		
		\item they do not represent any positive integer $\leq m$ globally.
	\end{enumerate}
	
	Let $\tilde S$ be the set of   primitive integral binary forms $F$ of degree $n$ that satisfy  \eqref{cGyory} and the following three conditions:
	\begin{enumerate}[(i)]
		\item $F$ has  Galois group $S_n$ (in particular $F$ is irreducible and maximal);
		
		\item $F$ has two simple roots in $\F_{p_i}$ for all
		$i=1,\ldots,k$;
		
		\item  For each prime $p>M$, $F$ does not factor as $cM(x,y)^r$ modulo $p$ for any $r>1$ and any binary form $M(x,y)$ and constant $c>0$.
	\end{enumerate}

	For each $F\in\tilde S$, we will construct $2^k-32n-11$ forms $G_j\in 
	S$ such that, if $H(F)<\tilde X$, then $H(G_j)<X$.
	Thus, if we denote by $h(S,H)$ the number of forms in $S$ with height $H$, we will have
	\begin{equation}\label{ineqclass0}
		\sum_{0<H<X} h(S,H) \geq (2^k-32n-11) \sum_{0<H<\tilde X} h(\tilde S,H).
	\end{equation}
	The right hand side of \eqref{ineqclass0} can be estimated as  in \cite[Theorem 5.3]{AB}
	\footnote{The estimation is done  using  
		Ekedahl's sieve (we refer to the references within \cite{AB} for an exposition) together with the facts that  100\% of forms satisfy (i) and 
		that conditions  (ii) and (iii) are described by  a suitable set of congruences.}.
	We have: 
	\begin{equation}\label{ineqclass2}
		\sum_{0<H<\tilde X} h(\tilde S,H)
		\sim 2^n \tilde X^{n+1} \prod_p\mu_p(\tilde S),
	\end{equation}
	where $\mu_p(\tilde S)$ denotes the $p$-adic density of $ \tilde S$ in the space of integral binary forms of degree $n$.
	By \eqref{ineqclass0} and \eqref{ineqclass2}, 
	\begin{equation}\label{ineqclass3}
		\sum_{0<H<X} h(S,H)\gg (2^k-32n-11) 2^n\Big(\dfrac{(n+1)^{k/2}}{2^{(n+1)k}\prod_{i=1}^kp_i^{n-1}}\Big)^{n+1} X^{n+1},
	\end{equation}
	where the implicit constant  depends only on $X$.
	
	The estimation \eqref{ineqclass3} already gives a positive proportion.  
	Using that $k\leq \frac{M}{\log M}$ and $\prod_{i=1}^k p_i \leq 4^M$, we have
	$$
	\dfrac{(n+1)^{k/2}}{2^{(n+1)k}\prod^k_{i=1} p_i^{n-1}} \gg \dfrac{1}{8^{M n}}. 
	$$
	
	
	For $p\in\left\{p_1,\ldots,p_k\right\}$, the density of forms satisfying (ii) is greater than
	$$
	\dfrac{\frac{p(p-1)}{2} (p^{n-1} -  p^{n-2})}{p^{n+1}}.
	$$
	If we set $p=p_i$, the  first term in the numerator counts the  possibilities for  two simple roots in $\F_{p_i}$  satisfying condition (ii),
	and the second term bounds the possibilities 
	for a binary form of degree $n-2$. The denominator counts the total number of binary forms modulo $p$ by considering
	the $p$ possible values for each of the $n+1$ coefficients. Hence the density of forms satisfying (ii) is greater than
	$$
	\dfrac{1}{2}\Big(  \dfrac{p-1}{p}\Big)^2.
	$$
	
	For $p>M$,  the density of forms satisfying (iii), i.e. that do not factor modulo $p$ as $cM(x,y)^r$ for any constant $c$, 
	any binary form $M$ and any $r>1$,  is
	\begin{align*}
		1-\dfrac{p-1}{p^{n+1}}  \sum_{\substack{r\mid n\\ r<n}} p^{r+1} &\geq 1- \dfrac{p-1}{p^n}  \sum_{r=0}^{[\frac{n}{2}]} p^r\\
		& = 1-\dfrac{p^{[\frac{n}{2}]+1}-1}{p^n}\\
		&\geq 1-\dfrac{1}{p^{[\frac{n}{2}]+1}}.
	\end{align*}
	The factor $p-1$ counts the possible values for $c$ and $[\cdot]$ denotes the integer part.
	Therefore, 
	\begin{align*}
		\prod_p \mu_p(\tilde S) &\geq \prod_{i=1}^k \dfrac{1}{2}\Big(  \dfrac{p_i-1}{p_i}\Big)^2 \prod_{p>M} 1-\dfrac{p^{[\frac{n}{2}]+1}-1}{p^n}\\
		&\geq  \dfrac{1}{2^k}\prod_{i=2}^{k+1} \Big(\dfrac{i-1}{i}\Big)^2  \prod_{p> M} 1-\dfrac{1}{p^{[\frac{n}{2}]+1}}\\
		&\geq \dfrac{1}{2^M(M+1)^2}  \prod_{p>M} 1-\dfrac{1}{p^{[\frac{n}{2}]+1}}.
	\end{align*}
	We conclude that
	$$
	\sum_{0<H<X} h(S,H) \gg  \dfrac{2^n}{8^{Mn(n+1)}(M+1)^2 }  
	\prod_{p> M} 1-\dfrac{1}{p^{[\frac{n}{2}]+1}}.
	$$
	
	\section{Construction of suitable forms}\label{secGj}
	
	\subsection{Construction of $2^k$ primitive forms}
	
	Let $F(x,y)=a_0x^n + a_1x^{n-1}y+\ldots+a_ny^n\in\Z[x,y]$ be a  primitive binary form of degree $n\geq3$.
	Suppose that $\alpha$ is a simple integer root of $F(x,1)$ modulo a prime $p$;
	\begin{equation}\label{alpha}
		F(\alpha,1)\equiv 0\pmod{p} \quad\mbox{with $\alpha\in\Z$}.
	\end{equation}
	Consider the form
	$$
	F_{\tiny{\begin{pmatrix}p &\alpha\\0 &1\end{pmatrix}}}(x,y)=F (px+\alpha y,y),
	$$
	and write $F (px+\alpha y,y)=e_0x^n+e_1x^{n-1}y+\ldots+e_ny^n$. 
	The coefficients $e_i$ are given by  
	\begin{equation}\label{en}
		e_{n-j}=p^j\sum_{i=0}^{n-j} a_i\alpha^{n-i-j}\begin{pmatrix}n-i\\j\end{pmatrix}.
	\end{equation}
	\begin{claim} The content of $F(px+\alpha y,y)$ is $p$. 
	\end{claim}
	\begin{proof}
		Any divisor of the content should divide $e_0=p^na_0$, and 
		thus it should divide either $p$ or $a_0$. By induction on the formula \eqref{en}, any divisor of the content should then divide $p$ or all 
		coefficients $a_0,\ldots,a_n$. Since $F$ is primitive, any divisor of the content should divide $p$. 
		If $j\geq 1$, clearly $e_{n-j}$ is divisible by $p$. By \eqref{alpha} we also have 
		$$e_n=F(\alpha,1)\equiv 0\pmod{p}.$$
		Since $e_{n-1}=pF_x'(\alpha,1)$  and $\alpha$ is a simple root modulo $p$, we have that
		$e_{n-1}\not\equiv 0\pmod{p^2}$. So the content is exactly $p$. 
	\end{proof}
	Hence the form
	\begin{equation}\label{Falpha}
		F_\alpha(x,y):= \dfrac{1}{p} F(px+\alpha y,y)
	\end{equation}
	has integral coefficients and content 1. For later, note that it is congruent to
	\begin{equation}\label{congr1}
		y^{n-1}L(x,y)\pmod{p}
	\end{equation}
	where $L(x,y)=F_x'(\alpha,1)x + y\frac{F(\alpha,1)}{p}$ satisfies $F_x'(\alpha,1)\not\equiv 0\pmod{p}$. 
	
	Note also that we can choose the integer $\alpha$ so that
	$$
	-\dfrac{p-1}{2}\leq \alpha\leq \dfrac{p-1}{2}.
	$$
	Then we have, by \eqref{en}, 
	\begin{align*}
		|e_{n-j}| &\leq p^j H(F)\left(\dfrac{p-1}{2}\right)^{n-j}\begin{pmatrix}n+1\\j+1\end{pmatrix}\\
		&<p^n H(F) \dfrac{2^{n+1}}{\sqrt{n+1}},
	\end{align*}
	where in the second inequality we used Stirling's approximation. Therefore, 
	$$
	H(F_\alpha) < p^{n-1} \dfrac{2^{n+1}}{\sqrt{n+1}} H(F).
	$$
	
	Now let $F\in S$. In particular, $F$ has two  (distinct) simple roots  modulo $p_i$ for all $i=1,\ldots,k$.
	Applying the  procedure described above with $p=p_1$, we obtain $F_{\alpha}$ and $F_{\beta}$, where $\alpha,\beta$ are  simple roots modulo $p_1$.
	The forms
	$F_{\alpha}$ and $F_{\beta}$ have again two distinct simple roots modulo 
	$p_i$ for $i=2,\ldots,k$, since the transformation matrices from $F$ to $F_{\alpha}$ and $F_{\beta}$ have determinant $p_1$, which is a unit in 
	$\F_{p_i}$.
	Hence we can apply the procedure described above to $F_{\alpha}$ and $F_{\beta}$ with $p=p_2$ and obtain four new forms with integral coefficients and
	content 1.
	Applying the whole procedure with each prime $p_i$,  we obtain $2^k$ primitive integral binary forms of degree $n$, 
	that we denote by $G_j$ ($j=1,\ldots,2^k$).
	Their discriminant and height satisfy respectively
	\begin{equation}\label{disc}
		D(G_j)= \Big(\prod_{i=1}^kp_i\Big)^{n-1} D(F),
	\end{equation} 
	$$
	H(G_j)\leq \Big(\prod_{i=1}^kp_i\Big)^{n-1}\dfrac{2^{(n+1)k}}{(n+1)^{k/2}} H(F).
	$$
	Since $F$ is irreducible over $\Q$ and all of 
	the matrix actions involved are rational, the forms $G_j$ are also irreducible over $\Q$.
	
	By construction, each solution to an equation 
	\begin{equation}\label{Gjh0}
		G_j(x,y)=h
	\end{equation}
	gives a solution to the equation
	\begin{equation}\label{Fh0}
		F(x,y)=h\prod_{i=1}^k p_i.
	\end{equation}
	Note that, with the notation \eqref{Falpha}, if $(x_0,y_0)$ is a solution to $F_\alpha(x,y)=h$ 
	and $F_\beta(x,y)=h$ simultaneously, since $\alpha\neq\beta$, then $(px_0+\alpha y_0,y_0)$ and $(px_0+\beta y_0, y_0)$ are two different solutions to $F(x,y)=ph$.
	Also,     if $F_\alpha(x,y)=h$ has two different solutions $(x_0,y_0)$ and $(x_0',y_0')$, then clearly $(px_0+\alpha y_0,y_0)$ and $(px_0'+\alpha y_0,y_0)$ are
	two different solutions to $F(x,y)=ph$.
	It follows that a solution to two equations
	$$
	G_j(x,y)=h,\qquad   G_{j'}(x,y)=h\qquad (j\neq j')
	$$
	gives at least two different solutions to \eqref{Fh0}.

	\subsection{The forms are not equivalent}
	Any two  $G_j,G_{j'}$ as above are of the form $F_\gamma$ and $F_{\gamma'}$ with
	$$
	\gamma = \begin{pmatrix} \prod_{i=1}^k p_i &\sum_{i=1}^k u_i\prod_{\ell=1}^{i-1}p_\ell\\0 &1\end{pmatrix}, \qquad 
	\gamma'= \begin{pmatrix} \prod_{i=1}^{k} p_i &\sum_{i=1}^{k}u'_i\prod_{\ell=1}^{i-1}p_\ell\\0 &1\end{pmatrix},  
	$$
	where $u_i,u'_i$ are simple roots modulo $p_i$. Any transformation $B$ such that $F_{\gamma B}=F_{\gamma'}$ satisfies 
	$B=\gamma^{-1}A\gamma'$, where $A$ stabilizes $F$. Since $F$ 
	has Galois group $S_n$, its stabilizer is trivial and so 
	$$B=\gamma^{-1}\gamma'=\begin{pmatrix} 1
		&\frac{\sum_{i=1}^k(-u_i+u'_i)\prod_{\ell={1}}^{i-1}p_\ell}{\prod_{i=1}^k p_i}\\0 &1\end{pmatrix}.$$
	Using inductively  that $p_i\mid (u_i-u_i')$ only if $u_i=u_i'$ for $i=1,\ldots,k$, we conclude that 
	the term in the top right of the matrix is an integer only if $\gamma=\gamma'$.
	Otherwise,  $B\not\in\mathrm{GL}(2,\Z)$. 
	
	\subsection{The forms are maximal}\label{max} 
	Suppose that a form $G_j$ is not maximal, so there is an integer matrix $A$ with
	$\det(A) > 1$ and $\det(A)\mid\disc(G_j)$. In particular, there is a prime $p$ such that $p\mid\det(A)$ and $p\mid\disc(G_j)$.
	For $p>p_k$,  up to a unit constant every form $G_j$ is equivalent to $F$ over $\Z_p$, and since $F$ is maximal,
	$G_j$ is also  maximal over $\Z_p$, so $p\nmid \det(A)$.
	Now let $p\in\left\{p_1,\ldots,p_k\right\}$. 
	In the construction  process of  $G_j$, there is a form $K(x,y)$ 
	that by \eqref{congr1} satisfies
	\begin{equation}\label{loc1}
		K(x,y)\equiv y^{n-1}(\ell_1 x+\ell_2y)\pmod{p}
	\end{equation}
	with $p\nmid\ell_1$. 
	By construction, $G_j$ is equivalent to $K$ over  $\Z_p$ up to a unit constant. 
	Since $\disc(K)\equiv \ell_1^{2(n-1)}\not\equiv0\pmod{p}$, $p\nmid\disc(G_j)$.

	\section{All forms represent locally every  integer $\leq m$}\label{sec local}
	
	Let $h$ be a positive integer, $h\leq m$. In this section we will show that every form $G_j$ represents $h$ over $\Z_p$ for all prime $p$.

	\subsection{At primes $\leq M$}
	Let $p$ be a prime in $\left\{p_1,\ldots,p_k\right\}$. 
	As in section \ref{max}, we consider the form $K(x,y)$ in the construction of $G_j$ 
	that by \eqref{congr1} satisfies
	\begin{equation}\label{loc1}
		K(x,y)\equiv y^{n-1}(\ell_1 x+\ell_2y)\pmod{p}
	\end{equation}
	with $p\nmid\ell_1$. 
	
	Set $p=p_i$, so that $K(x,y)$ is obtained in the $i-th$ step of the construction.  Using \eqref{loc1} it is clear that, for  each $y_0=1,\ldots,p-1$, 
	the equation 
	$$
	K(x,y)=h\prod_{\ell=i+1}^k p_\ell
	$$ has a solution modulo $p$ which is not a root  of the derivative $K_x'(x,y)\equiv y^{n-1}\ell_1$ ($\mbox{mod } p$). 
	By Hensel's Lemma, we can  lift the solution
	over $\Z_p$. Now by construction, $K(x,y)$ and $G_j(x,y)\prod_{\ell=i+1}^k p_\ell $ are equivalent over $\Z_p$, 
	so the equation $G_j(x,y)=h$ has also a solution over $\Z_p$.
	
	\subsection{At primes $>M$}
	Let $p$ be a prime greater than $M$. In particular, $p>h$ and $p>n$,  so $p\nmid hn$.
	Consider the  plane curve $h z^n=G_j(x,y)$ defined over  $\F_p$. 
	It follows from  (iii)  and the irreducibility of $G_j$, that 
	$G_j$ is irreducible also over $\F_p$.
	This ensures that $G_j(x,y)=0$ does not have any solution over $\F_p$ and hence the curve is smooth (we use here that $p\nmid hn$).
	Let $N$ be the number of points over $\F_p$ on the curve.
	Since $p>(2g+1)^2$ (where $g=\frac{(n-1)(n-2)}{2}$ was introduced earlier), the Hasse-Weil bound 
	$$
	|N-(p+1)|\leq 2g\sqrt{p}
	$$
	gives a smooth point $(x_0,y_0,z_0)$ with $z_0\neq0$ over $\F_p$. 
	This point gives a solution $(x_0z_0^{-1},y_0z_0^{-1})$ to $G_j(x,y)\equiv h\pmod{p}$ that can be lifted over $\Z_p$ by Hensel's Lemma.

	\section{Many forms  do not represent globally any  integer $\leq m$}\label{sec global}
	
	Let $h$ be a positive integer, $h\leq m$. Suppose that $G_j(x,y)=h$ has a solution $(x_0,y_0)\in\Z^2$. By construction,
	\begin{equation}\label{Gj}
		G_j= \Big(\prod_{i=1}^k p_i^{-1}\Big)
		F_{\tiny{\begin{pmatrix} p_1 &u_1\\0 &1\end{pmatrix}}\cdots\tiny{\begin{pmatrix} p_k &u_k\\0 &1\end{pmatrix}}},
	\end{equation}
	where $u_i$ are simple roots modulo $p_i$. It follows that
	$$
	(X_0,y_0):=\begin{pmatrix} p_1 &u_1\\0 &1\end{pmatrix}\cdots\begin{pmatrix} p_k &u_k\\0 &1\end{pmatrix} \begin{pmatrix} x_0\\y_0 \end{pmatrix}
	$$
	is a solution to 
	\begin{equation}\label{Fh}
		F(x,y)=h\prod_{i=1}^k p_i.
	\end{equation}
	Hence a solution to the inequality 
	\begin{equation}\label{Gm}
		G_j(x,y)\leq m
	\end{equation}
	gives a solution to 
	\begin{equation}\label{Fm}
		F(x,y)\leq  m \prod_{i=1}^k p_i
	\end{equation}
	and actually a primitive solution (a non-primitive solution $(a\tilde X_0,a\tilde y_0)$ to \eqref{Fm} with $(\tilde X_0,\tilde y_0)=1$  gives a primitive solution to
	$F(x,y)\leq \frac{m}{a^n} \prod_{i=1}^k p_i$).
	Suppose that $(x_0,y_0)$ is also a solution to $G_{j'}(x,y)=h'$, with $j\neq j'$, $0<h'\leq m$. Then $G_{j'}$ is of the form \eqref{Gj} with $u_1',\ldots,u_k'$ instead
	of $u_1,\ldots,u_k$ and
	$$
	(X_0',y_0):= \begin{pmatrix} p_1 &u'_1\\0 &1\end{pmatrix}\cdots\begin{pmatrix} p_k &u'_k\\0 &1\end{pmatrix} \begin{pmatrix} x_0\\y_0 \end{pmatrix}
	$$
	is a solution to 
	$$
	F(x,y)=h'\prod_{i=1}^k p_i.
	$$
	If $h=h'$, then we already know (see section 4.1) that $(x_0,y_0)$ gives at least two solutions to \eqref{Fh} and hence to \eqref{Fm}. 
	If $h\neq h'$, then $F(X_0',y_0)\neq F(X_0,y_0)$ so $X_0\neq X_0'$, and hence also in this case $(x_0,y_0)$ gives two solutions to \eqref{Fm}.
	Therefore, a solution to $\ell$ inequalities \eqref{Gm} (with $j\in\left\{1,\ldots,2^k\right\}$) gives at least $\ell$ primitive solutions to \eqref{Fm}. 
	By Theorem \ref{Gyory}, the inequality \eqref{Fm} has at most $32n+11$ primitive solutions. So there are at least $2^k-32n-11$ forms 
	$G_j$ with no solution 
	to \eqref{Gm}.


\end{document}